\documentclass{amsart}
\usepackage{amsfonts}
\usepackage{amssymb,latexsym}
\usepackage{enumerate}

\newtheorem{theorem}{Theorem}[section]
\newtheorem{lem}[theorem]{Lemma}
\newtheorem{cor}[theorem]{Corollary}
\newtheorem{prop}[theorem]{Proposition}
\newtheorem{rem}[theorem]{Remark}

\theoremstyle{definition}

\newcommand\gal{\mathrm{Gal}}

\theoremstyle{remark}

\numberwithin{equation}{section}

\begin{document}

\title[Galois groups and Genera of Quasi-cyclotomic Function Fields]{Galois groups and Genera of a kind of Quasi-cyclotomic Function Fields}

\author{Min Sha}
\address{Department of Mathematical Sciences, Tsinghua University, Beijing 100084, P.R. China}
\curraddr{} \email{shamin2010@gmail.com}
\thanks{}

\author{Linsheng Yin}
\address{Department of Mathematical Sciences, Tsinghua University, Beijing 100084, P.R. China}
\email{lsyin@math.tsinghua.edu.cn}
\thanks{This work was supported by NSFC Project No.10571097.}


\date{}

\dedicatory{}


\begin{abstract}
We call a $(q-1)$-th Kummer extension of a cyclotomic function field
a quasi-cyclotomic function field if it is Galois, but non-abelian,
over the rational function field with the constant field of $q$
elements. In this paper, we determine the structure of the Galois
groups of a kind of quasi-cyclotomic function fields over the base
field. We also give the genus formulae of them.
\end{abstract}

\maketitle

\section{Introduction}                                                            

We call a $(q-1)$-th Kummer extension of a cyclotomic function field
a quasi-cyclotomic function field if it is Galois, but non-abelian,
over the rational function field $k=\mathbb{F}_{q}(T)$. A large kind
of such fields were described explicitly in \cite{LY} following the works
in \cite{BGKY} and \cite{BY}. In this paper, we describe the Galois groups of
this kind of quasi-cyclotomic function fields by generators and
relations following the method in \cite{YZ} by using the results in \cite{BY}
and \cite{LY}. We also give the genus formulae of them.

Now we recall the constructions of the quasi-cyclotomic function
fields in \cite{LY}.

Let $k=\mathbb{F}_{q}(T)$ be the rational function
field over the finite field $\mathbb{F}_{q}$ of $q$ elements. In
this paper we always assume that the characteristic of $k$ is an odd
prime number $p$. Put $\mathbb{A}=\mathbb{F}_{q}[T]$. Let $\Omega$
be the completion of the algebraic closure of
$\mathbb{F}_{q}((1/T))$ at the place $1/T$. Let $k^{ac}$ be
the algebraic closure of $k$ in $\Omega$. Let $k^{ab}$ be the
maximal abelian extension of $k$ in $k^{ac}$.

Let $\bar{\pi} \in \Omega$ be the period of the Carlitz module,
namely the lattice $\bar{\pi}\mathbb{A}$ of rank one corresponds to
the Carlitz module. The Carlitz exponential function ${\bf e}_{C}$ is
defined by
\begin{equation}
{\bf e}_{C}(x)=x\prod\limits_{0\neq u\in
\bar{\pi}\mathbb{A}}\big(1-\frac{x}{u}\big), \quad x\in \Omega.
 \notag
\end{equation}
For $A\in \mathbb{F}_{q}((1/T))$, let $\{A\}$ be the representation
in $(\mathbb{F}_{q}((1/T))\setminus \mathbb{A})\cup \{0\}$
 of $A$ modulo $\mathbb{A}$, we define
\begin{equation}
{\rm sin}(A)=\sqrt[q-1]{-1}\cdot {\bf e}_{C}(\bar{\pi}\{A\}/{\rm
sgn}(\{A\})),
 \notag
\end{equation}
where sgn is a fixed sign function on $\mathbb{F}_{q}((1/T))$.
For the definition of sign function, see \cite[Definition 7.2.1]{Goss}.

Let $\mathcal{A}$ be the free ablian group generated by the symbols
$[A]$, $A\in k\setminus \mathbb{A}$. Define two homomorphisms
\begin{center}
sin, ${\bf e}$: $\mathcal{A} \to k^{ab*}$
\end{center}
such that ${\rm sin}([A])={\rm sin}(A)$ and ${\bf e}([A])={\bf e}_{C}(\bar{\pi}A)$ for
$A\not\in \mathbb{A}$, and ${\rm sin}([A])=1$ and ${\bf e}([A])=1$ otherwise.

Fix a total order $<$ in $\mathbb{A}$. Write $d_{A}$ for the degree
of $A \in \mathbb{A}$. Let $M \in \mathbb{A}$ be monic. Put
\begin{center}
$S_{M}=\Big\{$monic prime factors of $M\Big\}$.
\end{center}
Fix a generator $\gamma$ of $\mathbb{F}_{q}^{*}$. For $P,Q\in S_{M}$
with $P<Q$, let
\begin{equation}
{\bf a}_{PQ}=\sum\limits_{\substack{d_{A}<d_{Q} \\
A:{\rm monic}}}\sum\limits_{\substack{d_{B}<d_{P} \\
B:{\rm
monic}}}\sum\limits_{s=1}^{q-1}s\Big([\frac{BQ+\gamma^{-s}A}{PQ}]-[\frac{AP+\gamma^{-s}B}{PQ}]\Big).
\notag
\end{equation}
Notice that there is a print mistake in \cite{LY}, where $s$ runs from $1$ to $q-2$ in the definition of ${\bf a}_{PQ}$.
We also want to indicate that the homomorphism ${\bf e}$ gives the same value in the
${\bf a}_{PQ}$ here and in the ${\bf a}_{PQ}$ of \cite{BY}.

 We put
\begin{equation}
\begin{array}{lll}
 u_{PQ}=\left\{ \begin{array}{ll}
                 {\rm sin}{\bf a}_{PQ},  & \textrm{if $2|d_{P},2|d_{Q}$},\\
                  \\
                 \sqrt{P}{\rm sin}{\bf a}_{PQ},  & \textrm{if $2|d_{P},2\nmid d_{Q}$},\\
                 \\
                 \sqrt{Q}{\rm sin}{\bf a}_{PQ},  & \textrm{if $2\nmid d_{P},2|d_{Q}$},\\
                 \\
                 \sqrt{PQ}{\rm sin}{\bf a}_{PQ},  & \textrm{if $2\nmid d_{P},2\nmid d_{Q}$}.\\
                 \end{array} \right.
\end{array}
\notag
\end{equation}

Set $K=k({\bf e}_C(\frac{\bar{\pi}}{M}))$, which is the cyclotomic
function field of conductor $M$   whose
Galois group over $k$ is canonically isomorphic to $(\mathbb{A}/M\mathbb{A})^{*}$.
Since $u_{PQ}\in K$, put $\widetilde{K}=K(\sqrt[q-1]{u_{PQ}})$. By
\cite[Theorem 3]{LY}, $\widetilde{K}$ is a quasi-cyclotomic function field
over $k$, which implies that $[\widetilde{K}:k]=(q-1)\Phi(M)$, where $\Phi(M)$ is the number of
elements in $(\mathbb{A}/M\mathbb{A})^{*}$.

\section{The Galois groups}                                                                    
Let $G=\gal({K}/k)$ and
$\widetilde{G}=\gal(\widetilde{K}/k)$ be the Galois groups of
the extensions $K/k$ and $\widetilde{K}/k$ respectively. In this
section, we determine $\widetilde{G}$ by generators and relations.

In the sequel, we write $w=q-1$ and $u=u_{PQ}$ for simplicity.

First, we want to indicate a basic fact without proof. We will use it
several times without indication.
\begin{lem}
There exists $a\in \mathbb{F}_{q}^{*}$ such that
$\sin{\bf a}_{PQ}=a{\bf e}({\bf a}_{PQ})$.
\end{lem}

Clearly Gal$(\widetilde{K}/K)$ is isomorphic to $\mathbb{Z}/w\mathbb{Z}$.
Recall that $\gamma$ is a fixed generator of $\mathbb{F}_{q}^{*}$.
Let $\epsilon\in{\rm Gal}(\widetilde{K}/K)$ be a generator such that
$$
\epsilon(\sqrt[w]{u})=\gamma\sqrt[w]{u}.
$$
Denote by $\log_{\gamma}$ the isomorphism
$$
\log_{\gamma}: \quad\mathbb{F}_{q}^{*} \to
 \mathbb{Z}/w\mathbb{Z},\ \gamma^{i} \mapsto \bar{i}.
$$
Each element of $G$ has $w$ liftings in $\widetilde{G}$. Then we have a coarse
description about $\widetilde{G}$.

\begin{lem}\label{lift}
For any $\sigma \in G$, choosing $v_{\sigma}\in K^{*}$ such that                                
$\sigma(u)=v_{\sigma}^{w}u$, we can define a lifting $\tilde{\sigma} \in
\widetilde{G}$ of $\sigma$ by
$\tilde{\sigma}(\sqrt[w]{u})=v_{\sigma}\sqrt[w]{u}$. Then
$\widetilde{G}=\{\tilde{\sigma}\epsilon^{j}|\sigma \in G, 0\le j\le
w-1\}$, and the multiplication in $\widetilde{G}$ is given by
$\widetilde{\sigma\tau}=\tilde{\sigma}\tilde{\tau}\epsilon^{{\rm
log}_{\gamma}i(\sigma,\tau)}$, where
$i(\sigma,\tau)=\frac{v_{\sigma\tau}}{v_{\sigma}\sigma(v_{\tau})}\in
\mathbb{F}_{q}^{*}$. For any $\tilde{\sigma} \in \widetilde{G}$,
$\epsilon$ and $(\tilde{\sigma})^{w}$ belong to the center of
$\widetilde{G}$.
\end{lem}
\begin{proof}
By \cite[Section 5.1.2]{LY}, there exists such $v_{\sigma}\in K^{*}$ for any
$\sigma \in G$. The rest of the proof is trivial, we refer to
the proof of \cite[Lemma 1]{YZ}.
\end{proof}

 Let $M=P_{1}^{r_{1}}P_{2}^{r_{2}}\cdots
P_{n}^{r_{n}}$ be the prime decomposition of $M$. We have the
isomorphism:
\begin{center}
$G \cong (\mathbb{A}/M\mathbb{A})^{*} \cong
(\mathbb{A}/P_{1}^{r_{1}}\mathbb{A})^{*}\times
(\mathbb{A}/P_{2}^{r_{2}}\mathbb{A})^{*}\times \cdots \times
(\mathbb{A}/P_{n}^{r_{n}}\mathbb{A})^{*}$.
\end{center}
Different from the case of characteristic 0, now each
$(\mathbb{A}/P_{i}^{r_{i}}\mathbb{A})^{*}$ is not always cyclic. But
we have the decomposition
$(\mathbb{A}/P_{i}^{r_{i}}\mathbb{A})^{*}\cong
(\mathbb{A}/P_{i}^{r_{i}}\mathbb{A})^{(1)}\times
(\mathbb{A}/P_{i}\mathbb{A})^{*} $, where
$(\mathbb{A}/P_{i}^{r_{i}}\mathbb{A})^{(1)}$ is a $p$-group of order
$|P_{i}|^{r_{i}-1}$ and $(\mathbb{A}/P_{i}\mathbb{A})^{*}$ is a
cyclic group of order $|P_{i}|-1$, where $|P_{i}|=q^{d_{P_{i}}}$, see \cite[Proposition 1.6]{Rosen}.
For $1\le i\le n$, since the inertia group of $P_{i}$ in $K$ is isomorphic to $(\mathbb{A}/P_{i}^{r_{i}}\mathbb{A})^{*}$,
we choose a $\sigma_{P_i}\in G$ with
$\langle\sigma_{P_i}\rangle\cong (\mathbb{A}/P_{i}\mathbb{A})^{*}$ such that $\sigma_{P_i}$ is contained in
the inertia group of $P_{i}$. Then we have
\[
G=G^{(p)}\times G^{\prime},
\]
where $G^{\prime}=\langle\sigma_{P_1}\rangle\times
\cdots\times\langle\sigma_{P_n}\rangle$, and $G^{(p)}$ is the $p$-Sylow subgroup
of $G$. In fact,
$G^{(p)}\cong(\mathbb{A}/P_{1}^{r_{1}}\mathbb{A})^{(1)}\times\cdots\times(\mathbb{A}/P_{n}^{r_{n}}\mathbb{A})^{(1)}$.

 Let
$\widetilde{G^{(p)}}$ and $\widetilde {G^{\prime}}$ be the subgroups of
$\widetilde G$ consisting of all liftings of the elements in
$G^{(p)}$ and in $G^{\prime}$ respectively. It is easy to see that both of them are
normal subgroups of $\widetilde G$. Then we can get a decomposition of $\widetilde G$.
\begin{lem}    \label{product}                                                                    
Let $\widetilde{G}^{(p)}$ be the $p$-Sylow subgroup of $\widetilde
G$. Then
\[
\widetilde G^{(p)}\cong\widetilde{G^{(p)}}/\langle\epsilon\rangle\cong
{G^{(p)}}.
\]
Furthermore, we have $\widetilde G=\widetilde
{G}^{(p)}\times\widetilde {G^{\prime}}$, and $\widetilde{G}^{(p)}$ is contained in the center of $\widetilde G$.
\end{lem}
\begin{proof}
Since $|\widetilde{G}|=w|G|=|G^{(p)}|\cdot|\widetilde{G^{\prime}}|$, we
have $|\widetilde{G}^{(p)}|=|G^{(p)}|$. In addition, since the order of $\epsilon$ and $p$ are coprime,
for each element of $G^{(p)}$, there exists at most one lifting contained in $\widetilde{G}^{(p)}$.
So for each $\sigma\in G^{(p)}$, there exists a unique lifting $\sigma^{\prime}$ of $\sigma$ such that
$\sigma^{\prime}\in\widetilde{G}^{(p)}$. Then the map
$\sigma\mapsto\sigma^{\prime}\mod\langle\epsilon\rangle$ gives the isomorphism
$G^{(p)}\cong {\widetilde{G^{(p)}}}/\langle\epsilon\rangle$ and the map
$\sigma\mapsto\sigma^{\prime}$ gives the isomorphism
${G^{(p)}}\cong\widetilde{G}^{(p)}$.

Since $\widetilde G^{(p)}=(\widetilde G^{(p)})^{w}$, by Lemma \ref{lift} we see that
$\widetilde G^{(p)}$ is contained in the center of $\widetilde{G}$. So
$\widetilde G^{(p)}$ is a normal subgroup of $\widetilde G$. In addition, as
$|\widetilde{G}|=|\widetilde{G}^{(p)}|\cdot|\widetilde{G^{\prime}}|$ and
$\gcd(|\widetilde{G}^{(p)}|,|\widetilde{G^{\prime}}|)=1$, we have
$\widetilde{G}=\widetilde{G}^{(p)}\times\widetilde{G^{\prime}}$.
\end{proof}

Next we need to investigate the subgroup $\widetilde{G^{\prime}}$.

For each generator
$\sigma_{P_i}\in G^{\prime}$ ($1\le i\le n$), according to Lemma \ref{lift} we fix a lifting
$\tilde\sigma_{P_i}$ of $\sigma_{P_i}$ in $\widetilde G$ as follows.

If $P_{i}\ne P,Q$, we define
\begin{equation}
\tilde{\sigma}_{P_{i}}(\sqrt[w]{u})=\sqrt[w]{u}. \notag
\end{equation}
In fact, by \cite[Section 3.3, 4.3 and 5.1]{BY}, we have
$\sigma_{P_{i}}(u)=u$.

If $P_{i}=P$ or $Q$, we define
\begin{equation}
\tilde{\sigma}_{P_{i}}(\sqrt[w]{u})=v_{\sigma_{P_{i}}}\sqrt[w]{u},
 \notag
\end{equation}
where $v_{\sigma_{P_{i}}}\in K$ is given by
\begin{equation}
v_{\sigma_{P}}=\left(\sqrt[w]{(-1)^{d_{Q}}}{\rm sin}{\bf
c}_{\sigma_{P}}\right)^{-1} \quad {\rm and} \quad
v_{\sigma_{Q}}=\left(\sqrt[w]{(-1)^{d_{P}}}{\rm sin}{\bf
c}_{\sigma_{Q}}\right)^{-1},
 \notag
\end{equation}
here ${\bf c}_{\sigma_{P}}$ and ${\bf c}_{\sigma_{Q}}$ were defined
in \cite[Section 4.2.5]{BY}. By \cite[Section 3.4.2 and 3.4.3]{LY}, we have
$\frac{u}{\sigma_{P}(u)}=(\sqrt[w]{(-1)^{d_{Q}}}{\rm sin}{\bf
c}_{\sigma_{P}})^{w}$ with $\sqrt[w]{(-1)^{d_{Q}}}{\rm sin}{\bf
c}_{\sigma_{P}} \in K^{*}$.

Hence, we have
\[
\widetilde{G^{\prime}}=\langle\tilde\sigma_{P_1},\cdots,\tilde\sigma_{P_n},\epsilon\rangle.
\]

Now we study the relations among these generators of
$\widetilde{G^{\prime}}$. First $\epsilon$ commutes with each generator.
For $L,R \in S_{M}, L<R$, set
$\alpha_{LR}=\frac{\sigma_{L}(v_{\sigma_{R}})/v_{\sigma_{R}}}{\sigma_{R}(v_{\sigma_{L}})/v_{\sigma_{L}}}$.
By Lemma \ref{lift}, we have
$\tilde{\sigma}_{L}\tilde{\sigma}_{R}=\tilde{\sigma}_{R}\tilde{\sigma}_{L}\epsilon^{{\rm
log}_{\gamma}\alpha_{LR}}$.
By \cite[Section 3.5 The Log Wedge Formula, 3.6 The Auxiliary Formula and 5.1 The Main Formula]{BY}, we see that the
generators $\tilde\sigma_{P_i}$ commute with each other except for
the relation
\[
\tilde{\sigma}_{P}\tilde{\sigma}_{Q}=\tilde{\sigma}_{Q}\tilde{\sigma}_{P}\epsilon^{-1}.
\]
So in fact, $\widetilde{G^{\prime}}=\langle\tilde\sigma_{P_1},\cdots,\tilde\sigma_{P_n}\rangle.$

By definition, if $P_{i}\ne P,Q$, then we have
ord$(\tilde{\sigma}_{P_{i}})$=ord$(\sigma_{P_{i}})$. Finally, we
need to compute the orders of $\tilde{\sigma}_{P}$ and
$\tilde{\sigma}_{Q}$.

Let $L\in S_M$ and let $I_{L}$ be the inertia group of $L$ in $K$.
It is known that $I_{L}\cong(\mathbb{A}/L^{r}\mathbb{A})^{*}\cong
(\mathbb{A}/ L^{r}\mathbb{A})^{(1)}\times
(\mathbb{A}/L\mathbb{A})^{*} $, where $r$ is the maximal power of
$L$ such that $L^{r}|M$. We fix an inertia group $\widetilde{I}_{L}$ of $L$ in
$\widetilde{K}$. Let $\widetilde{L}$ be a prime ideal in
$\widetilde{K}$ above $L$ such that the inertia group
$I(\widetilde{L}/L)=\widetilde{I}_{L}$. Let $\widetilde{G}_{i}$ be
the $i$-th ramification group of $\widetilde{L}|L$, $i \ge -1$. Then
by \cite[III 8.6]{HS}, $\widetilde{G}_{0}=\widetilde{I}_{L}$,
$\widetilde{I}_{L}/\widetilde{G}_{1}$ is cyclic of order relatively
prime to $p$, and $\widetilde{G}_{1}$ is the unique $p$-Sylow
subgroup of $\widetilde{I}_{L}$ which is contained in the center of
$\widetilde{I}_{L}$ by Lemma \ref{lift}.

Put $G_{L}=\langle\sigma_{L}\rangle\subset I_{L}$ and
$\widetilde{G}_{L}=\widetilde{G_{L}}\bigcap \widetilde{I}_{L}$,
where $\widetilde{G_{L}}$ is the subgroup of $\widetilde{G}$
consisting of all liftings of the elements of $G_{L}$. Denote by $e_{L}$
the ramification index of any prime ideal of $\widetilde{K}$ lying above $L$ in the extension $\widetilde{K}/K$.
\begin{prop} \label{inertia}
$\widetilde{I}_{L}$ is an abelian group with $\widetilde{I}_{L}=\widetilde{G}_{1}\times \widetilde{G}_{L}$, where
$\widetilde{G}_{1}\cong(\mathbb{A}/
L^{r}\mathbb{A})^{(1)}$ and $\widetilde{G}_{L}$ is a cyclic group                
generated by a lifting of $\sigma_{L}$. Furthermore,
 all the liftings of $\sigma_{L}$ have the same order
$e_{L}\cdot {\rm ord}(\sigma_{L})$.
\end{prop}
\begin{proof}
Set $H=\langle\epsilon\rangle$. The canonical homomorphism $\widetilde{I}_{L}
\to I_{L}$, $\tilde{\sigma}\mapsto\tilde{\sigma}|_{K}$, induces an
isomorphism $\widetilde{I}_{L}/(\widetilde{I}_{L}\cap H) \cong
I_{L}$.  Since $\widetilde{K}$ is abelian over $K$,
$\widetilde{I}_{L}\cap H$ is the inertia group of $L$ in the field
extension $\widetilde{K}/K$ by \cite[Proposition 9.8]{Rosen}. So the order of
$\widetilde{I}_{L}\cap H$ is $e_{L}$, and thus $|\widetilde{I}_{L}|=e_{L}|I_{L}|$.
Noticing that $\widetilde{G}_{L}\cap H=\widetilde{I}_{L}\cap H$ and the above homomorphism also induces an
isomorphism $\widetilde{G}_{L}/(\widetilde{G}_{L}\cap H) \cong
G_{L}$, we have $|\widetilde{G}_{L}|=e_{L}|G_{L}|$.

Since $\widetilde{G}_{1}$ is contained in the center of $\widetilde{I}_{L}$
and $\widetilde{I}_{L}/\widetilde{G}_{1}$ is cyclic, we see that
$\widetilde{I}_{L}$ is abelian. Noticing that
$\gcd(|\widetilde{G}_{L}|,|\widetilde{G}_{1}|)=1$ and
$|\widetilde{I}_{L}|=|\widetilde{G}_{1}|\cdot |\widetilde{G}_{L}|$,
we have $\widetilde{I}_{L}=\widetilde{G}_{1}\times
\widetilde{G}_{L}$. So $\widetilde{G}_{L}\cong
\widetilde{I}_{L}/\widetilde{G}_{1}$ is cyclic. Since
$I_{L}=\widetilde{I}_{L}|_{K}=\widetilde{G}_{1}|_{K}\times
\widetilde{G}_{L}|_{K}=\widetilde{G}_{1}|_{K}\times G_{L}$ and
$\widetilde{G}_{1}|_{K} \cong \widetilde{G}_{1}/(\widetilde{G}_{1}
\cap H)=\widetilde{G}_{1}$, we have $\widetilde{G}_{1}\cong
(\mathbb{A}/ L^{r}\mathbb{A})^{(1)}$.

As $\widetilde{G}_{L}|_{K}=G_{L}$, there exists a lifting
$\sigma_{L}^{\prime}$ of $\sigma_{L}$ belonging to
$\widetilde{G}_{L}$. Since the order of $\epsilon$ is a factor of
the order of $\sigma_{L}$, all the liftings of $\sigma_{L}$ have the same order. If
the order of $\sigma_{L}^{\prime}$ is less than
$|\widetilde{G}_{L}|$, then it is easy to show that the order of
each element is also less than $|\widetilde{G}_{L}|$. But
$\widetilde{G}_{L}$ is cyclic. Hence $\sigma_{L}^{\prime}$ must be a
generator of $\widetilde{G}_{L}$.
\end{proof}

\begin{rem}
{\rm The extension $\widetilde{K}/k$ gives us an example of a non-abelian function field extension with abelian inertia groups.}
\end{rem}

If  $P_{i}\ne P$ and $Q$, then $e_{P_{i}}=1$. Now we need to calculate the
ramification indices $e_{P}$ and $e_{Q}$.

Let $R$ be a monic irreducible polynomial in $\mathbb{A}$ and           
$A\in\mathbb{A}$ be coprime to $R$. Recall that the $(q-1)$th
residue symbol $(\frac{A}{R})\in \mathbb{F}_{q}^{*}$ is defined by
\begin{equation}
(\frac{A}{R}) \equiv A^{\frac{|R|-1}{q-1}}\quad {\rm mod} \, R.
\notag
\end{equation}

Let $v_{P}$ be the additive valuation in $k^{ab}$ associated to
$P$ defined in \cite[Section 6]{BY}. Notice that the restriction of $v_{P}$ in
$k({\bf e}_C(\frac{\bar{\pi}}{P}))$ is the normalized valuation of
 $k({\bf e}_C(\frac{\bar{\pi}}{P}))$ associated to
$P$. By \cite[Proposition 6.2]{BY}, we have
\begin{equation}
 v_{P}({\bf e}({\bf a}_{PQ})) \equiv {\rm
log}_{\gamma}(\frac{Q}{P}) \quad {\rm and} \quad v_{Q}({\bf e}({\bf
a}_{PQ})) \equiv -{\rm log}_{\gamma}(\frac{P}{Q})  \quad ({\rm mod}
\, w).\notag
\end{equation}
In addition, $v_{P}(P)$ equals to the ramification index $|P|-1$
of $P$ in $k({\bf e}_C(\frac{\bar{\pi}}{P}))/k$. Thus
$v_{P}(\sqrt{P})\equiv \frac{w}{2}d_{P}$ (mod $w$). Similarly, we
have $v_{Q}(\sqrt{Q})\equiv \frac{w}{2}d_{Q}$ (mod $w$).

Furthermore, combining with the reciprocity law
$(\frac{Q}{P})=(-1)^{d_{P}d_{Q}}(\frac{P}{Q})$, see \cite[Theorem 3.5]{Rosen}, we have
\begin{equation}
 v_{P}(u) \equiv {\rm
log}_{\gamma}(\frac{P}{Q}) \quad {\rm and} \quad v_{Q}(u) \equiv
-{\rm log}_{\gamma}(\frac{Q}{P})  \quad ({\rm mod} \, w). \notag
\end{equation}

By \cite[III 7.3]{HS}, noticing that the valuations there are different from what we use here,
we have $e_{P}=\frac{w}{\gcd(w,v_{P}(u))}$ and
$e_{Q}=\frac{w}{\gcd(w,v_{Q}(u))}$. So
\begin{equation}
e_{P}=\frac{w}{\gcd(w,{\rm log}_{\gamma}(\frac{P}{Q}))}
\quad\text{and}\quad e_{Q}=\frac{w}{\gcd(w,{\rm
log}_{\gamma}(\frac{Q}{P}))}. \notag
\end{equation}

Finally we get the following theorem.                                             
\begin{theorem}\label{galois}
We have $\widetilde G=\widetilde {G}^{(p)}\times\widetilde {G^{\prime}}$,
where $\widetilde {G}^{(p)}$ is the p-Sylow subgroup of $\widetilde
G$ which is contained in the center of $\widetilde G$, and $\widetilde
{G^{\prime}}=\langle\tilde\sigma_{P_1},\cdots,\tilde\sigma_{P_n},\epsilon\rangle$.
The generators $\tilde\sigma_{P_i}$ and $\epsilon$ commute with each
other except for the relation
$\tilde{\sigma}_{P}\tilde{\sigma}_{Q}=\tilde{\sigma}_{Q}\tilde{\sigma}_{P}\epsilon^{-1}$.
In addition, for $P_i\ne P,Q$, we have {\rm
ord}$(\tilde{\sigma}_{P_i})$={\rm ord}$(\sigma_{P_i})$, and for
$P_i= P$ or $Q$, we have
\begin{equation}
{\rm ord}(\tilde{\sigma}_{P})=\frac{w}{\gcd(w,{\rm
log}_{\gamma}(\frac{P}{Q}))}\cdot {\rm ord}(\sigma_{P}) \quad and
\quad {\rm ord}(\tilde{\sigma}_{Q})=\frac{w}{\gcd(w,{\rm
log}_{\gamma}(\frac{Q}{P}))}\cdot {\rm ord}(\sigma_{Q}).
 \notag
\end{equation}
\end{theorem}

Notice that for each $1\le i\le n$, ord$(\sigma_{P_i})=\Phi(P_{i})$.


\begin{cor}
$\widetilde{K}/k$ is a solvable extension.
\end{cor}
\begin{proof}
Notice that the commutator subgroup of $\widetilde{G}$ is $\langle\epsilon\rangle$.
\end{proof}

\begin{cor}
In the extension $\widetilde{K}/K$, all ramified prime ideals of $K$ are tamely ramified.
\end{cor}

\begin{cor}
For any prime ideal $\mathfrak{p}$ of $K$ not above $P$ and $Q$, it is unramified in $\widetilde{K}/K$.
\end{cor}

\begin{cor}
The prime ideals of $K$ above $P$ (resp. $Q$) are unramified if and only if
$(\frac{P}{Q})=1$ (resp. $(\frac{Q}{P})=1$).
\end{cor}

In addition, all infinite primes of $K$ are unramified in $\widetilde{K}/K$,
see Lemma \ref{infinite}.

\begin{cor}
If $2|d_{P}d_{Q}$, then we have $e_{P}=e_{Q}$.
\end{cor}

\begin{cor}
Suppose $2\nmid d_{P}d_{Q}$. We have:

{\rm (1)} If $e_{P}=1$, then $e_{Q}=2$.

{\rm (2)} If $e_{P}=w$, then $e_{Q}=w$ or $\frac{w}{2}$. Moreover, $e_{Q}=w$ if and only if $4|w$.
\end{cor}
\begin{proof}
Notice that
\begin{equation}
 {\rm log}_{\gamma}(\frac{P}{Q})  \equiv
{\rm log}_{\gamma}(\frac{Q}{P})+\frac{w}{2}  \quad ({\rm mod} \, w). \notag
\end{equation}
\end{proof}
If we exchange the positions of $e_{P}$ and $e_{Q}$, the above corollary is also true.

\begin{cor}
Let $L$ be a monic irreducible polynomial in $\mathbb{A}$, then $L$ is ramified in
$\widetilde{K}/k$ if and only if $L|M$.
\end{cor}

\section{The genus formula}                                                   

In this section we compute the genus of $\widetilde K$. We calculate
it by using Hasse's genus formula on Kummer extensions, which states
that for an $m$-th Kummer extension $E/F$ of algebraic function
fields, where $m$ is relatively prime to the characteristic of
$F$, we have
\begin{equation}
g_{E}=1+\frac{m}{[\mathbb{F}_{E}:\mathbb{F}_{F}]}\left[g_{F}-1+\frac{1}{2}\sum\limits_{\mathfrak{p}\in
\mathbb{P}_{F}}\left(1-\frac{1}{e_{\mathfrak{p}}}\right){\rm
deg}\mathfrak{p}\right], \notag
\end{equation}
where $g_{E}$ and $g_{F}$ are the genus of $E$ and $F$ respectively,
$\mathbb{F}_{E}$ and $\mathbb{F}_{F}$ are the constant fields of $E$
and $F$ respectively, $\mathbb{P}_{F}$ is the set of primes of $F$,
and $e_{\mathfrak{p}}$ is the ramification index of $\mathfrak{p}$
in $E/F$, see \cite[III 7.3]{HS}.

Recall that $M$ has the prime decomposition
$M=P_{1}^{r_{1}}P_{2}^{r_{2}}\cdots P_{n}^{r_{n}}$.
For the genus of $K$, we quote a formula from \cite[Theorem 12.7.2]{Sa}.
\begin{theorem}
We have                                                            
\begin{equation}
g_{K}=\left[\frac{q-2}{2(q-1)}-1\right]\Phi(M)+\frac{1}{2}\sum\limits_{i=1}^{n}s_{i}d_{i}\Phi(M/P_{i}^{r_{i}})+1,
 \notag
\end{equation}
where $d_{i}=d_{P_{i}}$,
$s_{i}=r_{i}\Phi(P_{i}^{r_{i}})-q^{d_{i}(r_{i}-1)}$ and
$\Phi(M)=|(\mathbb{A}/M)^{*}|$.
\end{theorem}

\begin{lem}
The constant field of $\widetilde{K}$ is $\mathbb{F}_{q}$.                
\end{lem}
\begin{proof}
Since the constant field of $K$ is $\mathbb{F}_{q}$, it suffices to
show that $u \notin \mathbb{F}_{q}$.

Suppose that $u \in \mathbb{F}_{q}$. Then for any
$\sigma \in G$, $\sigma(u)=u$.
We can get a lifting $\tilde{\sigma}$ of $\sigma$ defined by $\tilde{\sigma}(\sqrt[w]{u})=\sqrt[w]{u}$. Hence
$\widetilde{G}$ is an abelian group. This leads to a contradiction.
\end{proof}

In Section 2 we have computed the ramification indices in
$\widetilde{K}/K$ of all finite primes of $K$. To calculate the
genus of $\widetilde{K}$, we need to compute those of the infinite
primes.
\begin{lem}\label{infinite}
The infinite primes of $K$ are unramified in $\widetilde{K}/K$.
\end{lem}
\begin{proof}
Let $k_{\infty}\subset\Omega$ be the completion of $k$ at the place
$1/T$. Let $K^{+}=K\cap k_{\infty}$ be the maximal real subfield of
$K$. By \cite[Section 4.3]{BY}, we know ${\rm sin}{\bf a}_{PQ}\in k_{\infty}$.
It is known that for any monic square-free
polynomial $f(T)$ in $\mathbb{F}_{q}[T]$ with even degree, we have
$\sqrt{f(T)} \in k_{\infty}$. So $u \in k_{\infty}$. Thus $u \in K^{+}$.

Let $E=K^{+}(\sqrt[w]{u})$. Then $\widetilde{K}=EK$ and
$[E:K^{+}]=w$. Let $\infty$ be an arbitrary infinite prime of $K^{+}$,
$\infty_{1}$ an infinite prime of $K$ above $\infty$, $\infty_{2}$ an
infinite prime of $E$ above $\infty$, and $\widetilde{\infty}$ an
infinite prime of $\widetilde{K}$ above $\infty_{1}$. By \cite[Theorem 12.14]{Rosen},
the ramification index $e(\infty_{1}/\infty)=w$. Then by
Abhyankar's Lemma, see \cite[III 8.9]{HS}, the ramification index
$e(\widetilde{\infty}/\infty)=w$. Since
$e(\widetilde{\infty}/\infty)=e(\widetilde{\infty}/\infty_{1})\cdot
e(\infty_{1}/\infty)$, we have $e(\widetilde{\infty}/\infty_{1})=1$.
Thus $\infty_{1}$ is unramified in $\widetilde{K}/K$. Since
$\widetilde{K}$ is Galois over $K$, all infinite primes of $K$ are
unramified in $\widetilde{K}/K$.
\end{proof}

Now we can get the genus formula of
$\widetilde{K}$.
\begin{theorem} We have                                                        
\begin{equation}
g_{\widetilde{K}}=1+w\left[g_{K}-1+\frac{1}{2}\left(1-\frac{1}{e_{P}}\right)
d_{P}\Phi(M/P^{r_{P}})+\frac{1}{2}\left(1-\frac{1}{e_{Q}}\right)d_{Q}\Phi(M/Q^{r_{Q}})\right],
\notag
\end{equation}
where $r_{P}$ and $r_{Q}$ are the maximal powers of $P$ and $Q$ such that $P^{r_{P}}|M$ and $Q^{r_{Q}}|M$
respectively, $e_{P}=\frac{w}{\gcd(w,{\rm log}_{\gamma}(\frac{P}{Q}))}$
and $e_{Q}=\frac{w}{\gcd(w,{\rm log}_{\gamma}(\frac{Q}{P}))}$.
\end{theorem}
\begin{proof}
By Hasse's formula, we have
\begin{equation}
g_{\widetilde{K}}=1+w\left[g_{K}-1+\frac{1}{2}\sum\limits_{\substack{{\rm prime}\,\mathfrak{p}\,{\rm in}\,K \\
\mathfrak{p}|P\, {\rm or}\,\mathfrak{p}|
Q}}\left(1-\frac{1}{e_{\mathfrak{p}}}\right){\rm
deg\,}\mathfrak{p}\right],
 \notag
\end{equation}
where the sum is over all the inequivalent primes above $P$ or $Q$,
 $e_{\mathfrak{p}}$ is the ramification index of $\mathfrak{p}$ in
$\widetilde{K}/K$, and for $\mathfrak{p}|P$, ${\rm
deg\,}{\mathfrak{p}}=f(\mathfrak{p}/P)d_{P}$, $f(\mathfrak{p}/P)$ is the residue class degree,
similarly for $\mathfrak{p}|Q$.

We assume that there are $g_{P}$ and $g_{Q}$ different prime ideals
in $K$ above $P$ and $Q$ respectively. Then
\begin{equation}
g_{\widetilde{K}}=1+w\left[g_{K}-1+\frac{1}{2}\left(1-\frac{1}{e_{P}}\right)g_{P}f_{P}d_{P}+\frac{1}{2}
\left(1-\frac{1}{e_{Q}}\right)g_{Q}f_{Q}d_{Q}\right],
 \notag
\end{equation}
where $f_{P}$ and $f_{Q}$ are the residue class degrees of $P$ and
$Q$ in $K/k$ respectively.
Since $g_{P}f_{P}=\Phi(M/P^{r_{P}})$ and
$g_{Q}f_{Q}=\Phi(M/Q^{r_{Q}})$, we get the formula.
\end{proof}

\section{Acknowledgement}
We are very grateful to the referee for his careful reading and many valuable suggestions.
Especially, we thank him for telling us the genus formula of the cyclotomic function field in \cite{Sa}.

\end{document}